\documentclass{article}
\usepackage{amsfonts,amssymb,amsmath}

\newenvironment{proof}[1][Proof]{\textbf{#1.} }{\hfill
  \rule{0.5em}{0.5em} \medskip}

\def\oti{\hfill \rule{.5em}{.5em}}

\newcommand{\cl}[1]{\overline{#1}}

\newtheorem{theorem}{Theorem}

\newtheorem{lemma}[theorem]{Lemma}

\newtheorem{cor}[theorem]{Corollary}

\newtheorem{defn}[theorem]{Definition}

\newcommand{\R}{\mathbb{R}}
\newcommand{\N}{\mathbb{N}}
\newcommand{\diam}{\text{diam}\, }

\begin{document}

\title{Spaces with a Finite Family of Basic 
Functions}

\author{Paul Gartside\footnote{\emph{Corresponding author} Department
    of Mathematics, University of Pittsburgh, Pittsburgh, PA 15260, USA, email:
    gartside@math.pitt.edu.} and Ziqin Feng\footnote{Department
    of Mathematics, University of Pittsburgh, PA 15260, USA} \\ 
{\sl Dedicated to Bob Heath on his retirement}}

\date{July 2008}

\maketitle

\begin{abstract} A $T_1$ completely regular  space $X$ is finite
  dimensional, locally compact and separable metrizable if and only if
  $X$ has  a finite basic family: functions $\Phi_1,...,\Phi_{n}$ such
  that for all $f\in C(X)$ there are $g_1,..., g_{n}\in C(\R)$ satisfying
$f(x)=\sum_{i=1}^{n}g_i(\Phi_i(x))$ for all $x\in X$.

This give the complete solution to four problems on basic functions
posed by Sternfeld. \end{abstract}

\section{Introduction}

The 13th Problem of Hilbert's celebrated list \cite{H} is usually
interpreted as asking whether every continuous real valued function
of three variables can be written as a superposition (i.e.
composition) of continuous functions of two variables. Kolmogorov gave
a strong positive solution (we write $C(X)$ for all continuous real
valued maps on a topological space $X$, and $C^*(X)$ for the subset
of bounded maps):
\begin{theorem}[Kolmogorov Superposition, \cite{Kol}]\label{KST}
For a fixed $n \ge 2$, there are $n(2n+1)$ maps $\phi_{pq} \in
C([0,1])$  such that every map $f \in C([0,1]^n)$ can be
written:
\[ f(\mathbf{x}) = \sum_{q=1}^{2n+1} g_q\left( \Phi_q
(\mathbf{x})\right) \qquad \text{where } \Phi_q(x_1, \ldots , x_n) = 
\sum_{p=1}^n \phi_{pq} (x_p),\] and the $g_q \in C(\R)$ are maps
depending on $f$. 
\end{theorem}
In addition to solving the superposition problem, Kolmogorov's theorem
says that the functions $\Phi_1, \ldots , \Phi_{2n+1}$ from $[0,1]^n$ to
the reals form a finite `basis' for {\sl  all} continuous real valued
maps from $[0,1]^n$. This is a very striking phenomena, leading to the
following natural definition of Sternfeld \cite{St}:
\begin{defn}           Let $X$ be a topological space. A family
  $\mathbf{\Phi}   \subseteq   C(X)$ is said to be 
 {\bf basic} (respectively, {\bf basic}$^*$) for $X$ if each  $f \in
 C(X)$ (respectively, $C^*(X)$) 
can be written:
\ $\displaystyle{f=\sum_{q=1}^{n} \left(g_q \circ \Phi_q\right)}$, 

for some  $\Phi_1, \ldots , \Phi_n$ in $\mathbf{\Phi}$ and
  `co-ordinate functions' $g_1, \ldots , g_n \in C(\R)$.
\end{defn}
Beyond their intrinsic interest, basic functions
have proved to be widely useful. Since the use of basic functions
reduces calculations of functions  simply to addition and evaluation
of a fixed finite family of  functions, applications  to numerical
analysis, approximation and function reconstruction are immediately
apparent. But other applications have emerged including to neural
networks $\cite{HN, HN2, Ku}$.

Extending the Kolmogorov Superposition Theorem, Ostrand \cite{Ost} showed that
every compact metric space of dimension $n$ has a basic family of size
$2n+1$. Subsequently Sternfeld \cite{St} showed that this basic family is
minimal in the sense that a compact metric space with a basic family
of size no more than $2n+1$ must have dimension $\le n$. Noting that
Doss \cite{Doss} had shown that Euclidean $n$-space, $\R^n$, has a
basic family of size $4n$ for $n \ge 2$, Sternfeld asked (in Problems
9--13 of \cite{St}) exactly which spaces have a finite basic family,
and whether the minimal size of a basic family on a space $X$ was
$2n+1$ where $n = \dim (X)$. Hattori \cite{Hat} showed that  every
locally compact, separable metrizable  space $X$ of dimension $n$ has
a finite basic$^*$ family of size $2n+1$. He asked whether the
restriction to locally compact spaces was necessary. Our Main Theorem
below gives a strong and complete solution to all these problems.

Since spaces with a {\sl finite} basic family are finite dimensional,
it seems plausible that spaces with a {\sl countable} basic family
would be countable dimensional. But we prove that if a space has a countable
basic family, then some finite subcollection is also basic, and so the
space is finite dimensional (and locally compact, separable
metrizable). To facilitate the proof, and provide full generality we
make the following definition allowing more general superposition
representations than a `basic' representation.
\begin{defn} 
Let $X$ be a topological space. A family
  $\mathbf{\Phi}   \subseteq   C(X)$ is said to be 
 {\bf generating} (respectively, {\bf generating}$^*$) for $X$ with
  respect to a `set of operations' $M$ of continuous functions
  mapping from subsets   of Euclidean space into subsets of Euclidean
  space, if each  $f \in 
 C(X)$ (respectively, $C^*(X)$) 
can be written as a composition of functions from $\mathbf{\Phi}$, $M$ and
   $C(\R)$.
\end{defn}
Note that a basic (respectively, basic$^*$) family is generating
(respectively, generating$^*$) with respect to $M=\{+\}$, and clearly
`generating' implies `generating$^*$'.

\begin{theorem}[Main Theorem]\label{main}  Let $X$ be $T_1$ and
  completely regular. 
Then the following are equivalent:

1) $X$ has a countable generating$^*$ family with respect to a
   countable set of operations,

2) $X$ has a finite basic family, and

3) $X$ is a finite  dimensional, locally 
compact and separable metrizable.

Further, a locally compact separable metrizable space $X$ has
dimension $\le n$ if and only if it has a basic family of size $\le 2n+1$.
\end{theorem}
By the preceding note, 2) $\implies$ 1) is immediate. 
In the next section (Section~\ref{rest}) we prove 1) $\implies$ 3), in
Section~\ref{cons} we establish 3) $\implies$ 2), and we justify
the `Further' claim characterizing dimension in Section~\ref{dim}.

\section{Restrictions Induced by Generating$^*$ Families}\label{rest}

In this section, all topological spaces are $T_1$ and completely regular. 

\begin{lemma}\label{emb} Let $X$ have a generating$^*$ family
  $\mathbf{\Phi}$ with respect to $M$. Then $e : X \to
 \R^\mathbf{\Phi}$
 defined by  $e(x)=(\Phi(x))_{\Phi \in \mathbf{\Phi}}$ is an embedding.
\end{lemma}

\begin{proof}
Clearly $e$ is continuous (each projection is a $\Phi$ in
$\mathbf{\Phi}$ which is continuous). It is also easy to see $e$ is
injective. Take distinct $x, x'$ in $X$. Pick $f \in C^*(X)$ such that
$f(x)=0$, $f(x')=1$. Represent $f$ as a composition of $\Phi_1 ,
\ldots , \Phi_n$ in $\mathbf{\Phi}$, members of $M$ and $C(\R)$. If
$e(x)=e(x')$ then
$\Phi_i(x)=\Phi_i(x')$ for all $i$, and so $f(x)=f(x')$, which is a
contradiction. 

It remains to show that the topology induced on $X$ by $e$ contains
the original topology. Since $X$ is completely regular it is
sufficient to check that for every $f \in C^*(X)$ the map $e(f) : e(X)
\to \R$ defined by $e(f)(\mathbf{x}) = f(e^{-1}(\mathbf{x}))$ is
continuous. But each $f \in C^*(X)$ can be written as a
composition of some $\Phi_1 , \ldots , \Phi_n$ in $\mathbf{\Phi}$ and
members of $M$ and $C(\R)$. Note that for each $i$ we have
$\Phi(e^{-1}(\mathbf{x}))= \pi_{\Phi_i} (\mathbf{x})$, where $\pi_{\Phi_i}$ is
the projection map of $\R^\mathbf{\Phi}$ onto the $\Phi_i$th co-ordinate.
Hence $e(f)=f \circ e^{-1}$ is the composition of continuous maps,
namely the $\pi_{\Phi_i}$s and functions in $M$ and $C(\R)$, and so is
 continuous as required.
\end{proof}

Since any subspace of $\R^{\mathbb{N}}$ is separable metrizable, and
any subspace of $\R^n$ is finite dimensional we deduce from
Lemma~\ref{emb}:
\begin{cor}\label{sm_fd} \ 

a) A space with a countable generating$^*$ family is separable
   metrizable.

b) A space with a finite generating$^*$ family is finite dimensional.
\end{cor}

A subspace $C$ of a space $X$ is said to be $C^*$-embedded in $X$ if
every $f \in C^*(C)$ can be extended to a continuous bounded real
valued function on $X$. In a normal space all closed subspaces are
$C^*$-embedded. Compact subspaces are always $C^*$-embedded.
We note the following easy lemma:
\begin{lemma}\label{subspace}
If $\mathbf{\Phi}$ is a generating$^*$ family (respectively,
  basic$^*$) for a space $X$ with
  respect to $M$, and $C$ is $C^*$-embedded in $X$ then
  $\mathbf{\Phi}|C = \{ \Phi|C : \Phi \in \mathbf{\Phi}\}$ is a
  generating$^*$ (respectively, basic$^*$) family for $C$.
\end{lemma}

\begin{lemma}\label{lc} A  space with a countable generating$^*$
  family is locally compact. 
\end{lemma}

\begin{proof} Suppose the space $X$ has a countable  generating$^*$ family
$\mathbf{\Phi}$ with respect to $M$, but is not locally compact. Since
  $X$ is metrizable, it follows that the metric fan $F$ (defined
  below) embeds as a closed subspace in $X$. Hence by
  Lemma~\ref{subspace} it suffices to show that $F$ does not admit a
  countable 
  generating$^*$ family (with respect to any set of operations $M$).

The metric fan $F$ has underlying set $\{*\} \cup \left(\mathbb{N} \times
\mathbb{N}\right)$ and topology in which all points other
than $*$ are isolated and $*$ has basic neighborhoods $B(*,N) =\{*\}
\cup \left( [N,  \infty) \times \mathbb{N} \right)$. For a
  contradiction, let $\Phi=\{\Phi_1, \Phi_2, \ldots\}$ be a countable
  generating$^*$ family with respect to  $M$.

For each $i$, let $y_i=\Phi_i(x_0)$, and pick basic open $U_i$
containing $*$ such that $\Phi_i(U_i) \subseteq (y_1-1,y_1+1)$. Now
for each $n$ let $V_n = \bigcap_{i=1}^n U_i$. So $\Phi_i(V_n)
\subseteq (y_i-1,y_i+1)$ for $i=1, \ldots , n$. We can write $V_n =
\{*\} \cup \left([N_n,\infty) \times \mathbb{N}\right)$ and suppose,
  without loss of generality, that $N_n > N_m$ if $n > m$.

Fix $n$. Let $D^0= \{x_k^0= (N_n,k): k \in \mathbb{N}\}$. 
As $\{\Phi_1(x_k^0)\}_{k\in \N}$ is a subset of $[y_1-1,y_1+1]$,
which is sequentially compact, there is a $D^1=\{x_k^1:k\in
\N\}\subseteq D^0$ such that $\{\Phi_1(x_k^1)\}_{k\in\N}$ is
convergent. As $\{\Phi_2(x_k^1)\}_{k\in \N}$ is a subset of
$[y_2-1,y_2+1]$, which is sequentially compact, there is a
$D^2=\{x_n^2:n\in \N\}\subseteq D^0$ such that
$\{\Phi_2(x_k^2)\}_{k\in\N}$ is convergent. Inductively we get
$D^n=\{x_k^n:k\in\N\}$, which is infinite closed discrete and for
each $i=1,...,n$ the sequence $\{\Phi_i(x_k^n)\}_{k\in\N}$ is
convergent, say to $z_i^n$. Define $D^n_O = \{x_{2k-1}^n : k \in
\mathbb{N}\}$ and $D^n_E = \{ x_{2k}^n : k \in \mathbb{N}\}$.

Define $f: F \to [0,1]$ by: $f$ is identically zero outside $\bigcup_n
D^n_O$ (in particular, $f$ is zero on each $D^n_E$), and $f$ is
identically $1/n$ on $D^n_O$. Then $f$ is continuous and bounded.

Hence, for some $m$,  $f$ can be written as the composition of 
$\Phi_1, \ldots , \Phi_m$ and members of $M$ and $C(\R)$. Now, on the
one hand $\lim_k \Phi_i (x_{2k-1}^m)=z_{i.m} = \lim_k \Phi_i
(x_{2k}^m)$ so by continuity of the elements of $M$ and $C(\R)$ in the
compositional representation of $f$, $\lim_k f(x_{2k-1}^m) = \lim_k
f(x_{2k}^m)$, and on the other hand, $\lim_k f(x_{2k-1}^m) = 1/m \ne 0
= \lim_k f(x_{2k}^m)$. This is our desired contradiction.
\end{proof}

Let $Y$ be a locally compact separable metrizable space. 
Write $C_k(Y)$ for $C(Y)$ with the compact-open topology. Then
$C_k(Y)$ is a Polish (separable, completely metrizable) group. In
particular, for any $n$, $C_k(\R)^n$ is a Polish group.
 
\begin{lemma}\label{fg}
If $X$ has a countable generating$^*$ family with respect to a
countable set of operations, $M$, then $X$ has a finite generating$^*$
family with respect to a finite set of operations $M'$.
\end{lemma}
\begin{proof}
Let $\Phi_1, \Phi_2, \ldots$ be a countable generating$^*$ family for
$X$ with respect to the countable set of operations $M$. By
Lemma~\ref{lc} $X$ is locally compact and $C_k(X)$ is a Polish group.

Let $g_1,
g_2, \ldots$ be formal letters representing functions from $\R$ to
$\R$. Let $\mathcal{W}$ be the set of all formal compositions of
$\Phi_i$s, elements of $M$ and $g_i$s. Note that $\mathcal{W}$ is
countable.

Fix $w$ in $\mathcal{W}$. Then $w$  induces a map $(g_1, \ldots , g_n) \mapsto
w(g_1,\ldots , g_n)$ from $C_k(\R)^n \to C_k(X)$ where we substitute
actual $g_i \in C(\R)$ for the corresponding formal letter. This map
is continuous with respect to the compact-open topology.
Let $F_w = w(C_k(\R)^n)$. It is analytic. Define $G_w = F_w \cap
C_k(X,(0,1))$. Since $C_k(X,(0,1))$ is homeomorphic to $C_k(X)$ it is
Polish, and hence must be a $G_\delta$ subset of $C_k(X)$. So $G_w$ is
analytic in $C_k(X,(0,1))$.

Note, by the generating$^*$ property,  that $C_k^*(X) \subseteq
\bigcup_{w \in \mathcal{W}} F_w$. Hence $C_k(X,(0,1)) = \bigcup_{w \in
  \mathcal{W}} G_w$. By the Baire Category theorem there must be some
particular $w$ in $\mathcal{W}$ such that $G_w$ is not meager.

Fix a homeomorphism $h:\R \to (0,1)$. Via $h$, addition and
subtraction on $\R$ induce (continuous) group operations $\oplus ,
\ominus : (0,1) \times (0,1) \to (0,1)$. These operations on $(0,1)$
in turn induce operations on $C_k(X,(0,1))$ making this space a Polish
group.

Let $H_w$ be the subgroup of $C_k(X,(0,1))$ generated by $G_w$. By
Pettis' Theorem \cite{Pettis}, since $G_w$ is non-meager and analytic,
$G_w \ominus G_w$ has non-empty interior. Hence the subgroup $H_w$ is
open, and so coincides with $C_k(X,(0,1))$ (which is connected). 

Set $\mathbf{\Phi}'$ to be the finite set of $\Phi_i$s appearing in
$w$, and set $M'$ to be $\oplus, \ominus$ and the finite set of
elements of $M$ appearing in $w$.  Since $H_w=C(X,(0,1))$, each element
of $C(X,(0,1))$ is a composition of members of $\mathbf{\Phi}'$, $M'$
and $C(\R)$.

We check $\mathbf{\Phi}'$ is a finite
generating$^*$ family with respect to $M'$. For if $f \in C^*(X)$,
then $f$ maps into some open interval $(a,b)$. Fix a homeomorphism
$g_0 : \R \to \R$ taking $(0,1)$ to $(a,b)$. Then $f = g_0 \circ \left(
g_0^{-1} \circ f\right)$, where $g_0^{-1} \circ f$ is in
$C_k(X,(0,1))$. Hence $g_0^{-1} \circ f$ can be expressed as a
composition of elements of $\mathbf{\Phi}'$, $M'$ and some $g_1,
\ldots , g_n$ in $C(\R)$. But now $f$ is $g_0$ of this composition and
so is also expressible in terms of elements of $\mathbf{\Phi}'$, $M'$
and $C(\R)$, as required.
\end{proof}

We note that the finite generating$^*$ family is a subset of the
original family, and also that if the original family is generating
then we can take $M' \subseteq M \cup \{+,-\}$.

\medskip

\begin{proof}[Proof of 1) $\implies$ 3)]

Let $X$ be a space with a countable generating$^*$ family with respect
to a countable set of operations. By Corollary~\ref{sm_fd}~a) $X$ is
separable metrizable. 
Lemma~\ref{lc} then says that $X$ is locally compact. From
Lemma~\ref{fg} we deduce that $X$  has a finite
generating$^*$ family. Hence by  Corollary~\ref{sm_fd}~b) $X$ is finite
dimensional. 
\end{proof}

\section{Construction of Finite Basic Families}\label{cons}
This section is devoted to proving:
\begin{lemma}~\label{cons_basic} If $X$ is a locally compact,
  separable metrizable space of dimensiosn $\le n$ then $X$ has a
  basic family of size $2n+1$.
\end{lemma}
Lemma~\ref{cons_basic} is the forward implication of the `Further' claim in the
Main Theorem. The implication `3) $\implies$ 2)' of the Main Theorem
then follows. 

Recall that Hattori \cite{Hat} showed that  every
locally compact, separable metrizable  space $X$ of dimension $n$ has
a finite basic$^*$ family of size $2n+1$. Lemma~\ref{cons_basic} and
its 
proof  improves on Hattori's result and proof because: (1) it applies to all
functions (not necessarily bounded), (2) it is constructive (Hattori's
argument uses a Baire category argument) and (3) it is considerably
shorter than Hattori's. The proof is similar to that of Ostrand for
{\sl compact} metric  spaces. However difficulties arise because
continuous real valued functions on a locally compact space need not
be {\sl bounded}.

For this section, fix a locally compact, separable space
$X$ of dimension 
$\le n$, and with compatible metric $d$. We can find $\{K_m: m\ge
-1\}$  a countable cover of $X$ by compact sets such that
$K_{-1}=K_0=\emptyset$ and $K_m\subseteq K_{m+1}^0$ for each $m \ge
-1$. 
For each $m\ge 0$ we put $H_m=K_m\setminus K_{m-1}^\circ$, 
and set $U_m=K_{m+1}^\circ\setminus K_{m-1}$.
Since Ostrand has done the compact case, we can assume that the
$K_m$'s are {\sl strictly} increasing. We show $X$ has a basic family
of size $2n+1$.

The basic functions $\Phi_i$ are defined to be the limit of
approximations $f_k^i$. The approximations are defined inductively
along with some families of `nice' covers. These `nice' covers come
from  Ostrand's \cite{Ost} characterization of dimension: 
\begin{theorem}[Ostrand's Covering Theorem]  A metric space $Y$ of
  dimension $\leq n$ if and only 
if for each open cover $\mathcal{C}$ of $Y$ and each integer $k\geq
n+1$ there exist $k$ discrete families of open sets $\mathcal{U}_1,
\ldots , \mathcal{U}_k$ such that the union of any $n+1$ of the
$\mathcal{U}_i $ is a cover of $Y$ which refines $\mathcal{C}$.
\end{theorem}

\begin{lemma}\label{ind_step}
 Let $\gamma>0$. There are $2n+1$ many families
  $\mathcal{S}^1, \ldots , \mathcal{S}^{2n+1}$ of 
open subsets of $X$, and $\eta^m >0$  for $m \ge 0$, satisfying:

  (1) Each $\mathcal{S}^i$ is discrete in X.

  (2) For $k$ fixed and each $x\in X$ fixed, $|\{S\in
\bigcup_{i=1}^{2n+1}\mathcal{S}^i:x\in S\}|\geq   n+1$.

  (3) $\diam S<\gamma$ for any $S\in \bigcup_{i=1}^{2n+1}\mathcal{S}^i$.

  (4) $\bigcup_{i=1}^{2n+1}\mathcal{S}^i$ refines $\{U_m: m\in
  \omega\}$.

  (5) For any $m\in \N$, $\{S: S\in\bigcup_{i=1}^{2n+1}\mathcal{S}^i,
  S\cap K_m\neq \emptyset\}$ is   finite.

  (6) $S(H_m,\eta^m)\cap S=\emptyset \text{ if }
 H_m\cap \cl{S}=\emptyset \text{ for any  }
 S\in\bigcup_{i=1}^{n+1}\mathcal{S}^i$.  
 
  (7) $\cl{S(H_{m-1},\eta^{m-1})}\cap
\cl{S(H_{m+1},\eta^{m+1})}=\emptyset$.

\noindent In (6) and (7),  $S(H_m,\eta^m)=\{x \in X : d(H_m,x)\leq  \eta^m\}$ 
\end{lemma}

\begin{proof}  Let $\mathcal{C}=\{C_a: a\in \N\}$ be a
locally finite open cover of $X$ with: $\diam (C_a)< \gamma$ and
$|\{H_m: H_m \cap \cl{C_a}\neq \emptyset\}|\leq 2$, for each $a\in\N$.
Then by Ostrand's covering theorem, there exist $2n+1$ discrete
families of open sets $\mathcal{S}_1, \cdots, \mathcal{S}_{2n+1}$
which refines $\mathcal{C}$. Also the union of any $n+1$ of the
$\mathcal{S}_i$ is a cover of $X$.  So 1)-4) are easy to verify.

Fix $i$ with $1\leq i\leq 2n+1$. As $\mathcal{S}^i$ is discrete,
 $\{S: S\cap K_m\neq \emptyset, S\in \mathcal{S}^i\}$ is
finite. Thus condition 5) is satisfied.

Now fix $i$ and $m$, the discreteness of $\mathcal{S}^i$ guarantees
that $$H_m\cap \cl{\bigcup \{S: S\in \mathcal{S}^i \text{ and }
H_m\cap \cl{S}=\emptyset\}} =\emptyset.$$ So $d(H_m,\cl{\bigcup \{S:
S\in \mathcal{S}^i \text{ and } H_m\cap \cl(S)= \emptyset\} }>0$.
Then we can pick $\eta_i^m$ such that $S(H_m,\eta_i^m)\cap
S=\emptyset \text{ if }
 H_m\cap \cl{S}=\emptyset \text{ for any  } S\in\mathcal{S}^i$. Let
$\eta^m=\min\{\eta_i^m:i=1,\cdots,2n+1\}$. This satisfies 6).

Notice that since $H_m$ is compact for each $m\in\N$, we can pick $\eta^m$ small
enough such that $\cl{S(H_{m-1},\eta^{m-1})}\cap
\cl{S(H_{m+1},\eta^{m+1})}=\emptyset$, giving (7).
\end{proof}

\paragraph{Step 1: Construction of the approximations}

Again, we generalize the construction of Ostrand, but must find ways
around the problem of not having {\sl bounded} functions. 

By induction on $k \ge 0$, using Lemma~\ref{ind_step}, for
$i=1,...,2n+1$,  
there exist: positive real numbers $\epsilon_k$ with $\epsilon_1<1/4$,
$\gamma_k$, $\eta_k^m$ distinct positive prime numbers $r_k^{i}$,  discrete
families $\mathcal{S}_k^1,..., \mathcal{S}_k^{2n+1}$ and continuous
functions $f_k^i: X\rightarrow [0,k+1]$, with the following properties. 

For each $k\in \N$, the families  $\mathcal{S}_k^1,...,
\mathcal{S}_k^{2n+1}$, $\gamma_k$ and $\eta_k^m$ satisfy (1)--(7) of
Lemma~\ref{ind_step}. Further:

   (A) $\lim_{k\rightarrow \infty} \gamma_k=\lim_{k\rightarrow
   \infty}\epsilon_k=0$;

   (B) $\epsilon_k<1/\Pi_{i=1}^{2n+1} r_{k}^{i}$;

   (C) $f_k^i$ is constant on the closure of those 
members of $\mathcal{S}_k^i$ which
   have nonempty intersection with $K_m$ for $(m\leq k)$, the constant being an
   integral multiple of $1/r_k^i$, and takes different values on
   distinct members. Then we can take a continuous extension of $f_k^i$
    to   the rest of the space.

   (D) For any $S$ in $\mathcal{S}_k^i$ having nonempty intersection
   with $H_m$, $m-1<f_m^i(S)<m+1$. Also for $m\geq 2$, by (7), we can make
   $m-1<f_k^i(S(H_m,\eta_k^m)<m+1$. For each $i\in\N$, if $\cl{S}\cap
   H_m\neq \emptyset$ and 
   $\cl{S}\cap H_{m+1}\neq \emptyset$, then $m<f_m^i(C)<m+1$;if
   $\cl{S}\cap H_m\neq \emptyset$ 
   and  $\cl{S}\cap H_{m-1}\neq \emptyset$, then $m-1<f_m(S)^i<m$;

   (E) For each $\ell<j<k$ and $x\in K_\ell$,
   $f_j^i(x)<f_k^i(x)<f_j^i(x)+\epsilon_j-\epsilon_k$ for any $i$.

\paragraph{Step 2: Construction of the basic functions}
From (E), for any $x\in K_m$ and $k>m$,
$f_m^i(x)<f_k^i(x)<f_m^i(x)+\epsilon_1$ for any $i=1,\ldots, 2n+1$.
Thus we can take the uniform limit of $f_k^i$ restricted on $K_m$.
For any $x\in K_m$ let
$\Phi_{i}(x)=\lim_{k\rightarrow\infty}f^{i}_k(x)$. So $\Phi_{i}$ is
continuous on $K_m$ for each $m$. Hence $\Phi_i$ is
continuous on $X$. Also by (D) for $x\in H_m$,
$m-1<\Phi_{i}(x)<m+1+1/4$.

Let $\mathcal{V}_k^i=\{\Phi_i(S): S\in \mathcal{S}_k^i\}$. Then if
$S\cap K_m\neq \emptyset$ and $S\in \mathcal{S}_k^i$ with $k>m$,
$\Phi_i(S)$ is contained in the interval
$[f_k^i(S),f_k^i(S)+\epsilon_k ]$ by (E). By (B), these closed
intervals are disjoint for each fixed $m$ and $k$ with $k\geq m$.
Then each $\mathcal{V}_k^i$ is discrete.

\paragraph{Step 3: Construction of the coordinate functions}

Take any function $f\in C(X)$. We find $g_1, \ldots , g_{2n+1} \in
C(\R)$ such that $f = \sum_{i=1}^{2n+1} g_i \circ \Phi_i$.

For each $s \ge 0$, define the compact subset $L_s=K_{s+1}\setminus
K_{s-1}^{\circ}$. 
Since $K_1$ is compact and
$K_1\subseteq K_2^{\circ}$, there exists a function $f_1$ such that
$f_1(x)=f(x)$ for $x\in K_1$ and $f_1(x)=0$ for $x\in X\setminus
K_2^{\circ}$. Then letting $g_1=f-f_1$, it is easy to see that $g_1(x)=0$ for
$x\in K_1$. Similarly, there exists $f_2$ such that $f_2(x)=g_1(x)$
for $x\in K_2$ and $f_2(x)=0$ for $x\in X\setminus K_3^{\circ}$.
Inductively, $f$ can be written  as an infinite sum
$\sum_{s=1}^\infty f_s$ such that $f_s(x)=0$ for $x\in X\setminus
L_s$.

For each $s$, $f_s$ is bounded and uniformly continuous. Fix $s\in
\N$.  Note that for each $x\in L_s$, $s-2<\Phi_i(x)<s+2+1/4$.

By construction, if we restrict the discrete families
$\mathcal{S}_1, \cdots, \mathcal{S}_{2n+1}$ and the  functions $\Phi_1, \cdots,
\Phi_{2n+1}$ to $K_{s+1}$, 
then the discrete families and functions are exactly those defined
by Ostrand~\cite{Ost}.

In particular, the functions $\Phi_1|L_s, \ldots ,
\Phi_{2n+1}|L_s$ are {\sl basic} for $L_s$ (Lemma~\ref{rest}).  
Thus we can represent
$f_s|L_s(x)=\sum_{i=1}^{2n+1}g_i^s(\Phi_{i}|L_s(x))$, for some $g_i^s\in C(\R)$.
We can redefine  $g_i^s$ to be constantly zero outside of
$[s-2,s+2+1/4]$  because the
image of $\Phi_i$ is contained in $[s-2, s+2+1/4]$ and $f_s(x)=0$ if $x\in
L_s\setminus (L_s)^{\circ} $. Now $f_s=\sum_{i=1}^{2n+1} g_i^s \circ \Phi_i$.

Finally, letting $g_i=\sum_{s=1}^\infty g_i^s$, we see that $g_i$ is continuous
because $g_i(x)$ is a finite sum of non-zero continuous functions for each
$x\in \R$, and $f = \sum_{i=1}^{2n+1} g_i \circ \Phi_i$ -- as
required. $\oti$

\section{Characterization of Dimension}\label{dim}

Lemma~\ref{cons_basic} says that a locally compact, separable
metrizable space of dimension $\le n$ has a basic family of size $\le
2n+1$, giving the forward implication in the `Further' of
Theorem~\ref{main}. For the converse:
\begin{lemma} A  space $X$ with a basic$^*$ family $\Phi_1, \ldots ,
  \Phi_N$, where $N \le 2n+1$, has dimension $\le n$.
\end{lemma}

\begin{proof} Take any compact subset $K$ of $X$. By
  Lemma~\ref{subspace},  the maps $\Phi_1 |K, \ldots, \Phi_N|K$ form a
  basic$^*$   family for $K$, hence by compactness a basic family. 
By Sternfeld's result connecting dimension and basic families in
  compact spaces, it follows that $\dim K \le n$.

By Lemma~\ref{lc}, $X$ is locally compact, separable
metrizable. Hence it has a locally finite cover by compact sets --
each, by the above, of dimension $\le n$. By
the Locally Finite Sum Theorem for dimension, we deduce that $X$
itself must have dimension $\le n$.
\end{proof}

\end{document}